\documentclass[11pt]{article}
\usepackage{amssymb,amsmath,amsthm}
\usepackage{graphicx}
\usepackage{array,multirow}

\newtheorem{theorem}{Theorem}
\newtheorem{lemma}[theorem]{Lemma}

\title{Complete colorings of planar graphs}
\author{G. Araujo-Pardo\footnotemark[1]
\and F. E. Contreras-Mendoza\footnotemark[2]
\and S. J. Murillo-Garc{\' i}a \footnotemark[2]
\and A. B. Ramos-Tort \footnotemark[2]
\and C. Rubio-Montiel\footnotemark[3]}

\begin{document}
\maketitle

\def\thefootnote{\fnsymbol{footnote}}
\footnotetext[1]{Instituto de Matem{\' a}ticas at Universidad Nacional Aut{\' o}noma de M{\' e}xico, 04510, Mexico City, Mexico. E-mail address: {\tt garaujo@matem.unam.mx}.}
\footnotetext[2]{Facultad de Ciencias at Universidad Nacional Aut{\' o}noma de M{\' e}xico, 04510, Mexico City, Mexico. E-mail addresses: {\tt [esteban.contreras.math|jani.murillo|ramos{\_}tort]@ciencias.unam.mx}.}
\footnotetext[3]{Divisi{\' o}n de Matem{\' a}ticas e Ingenier{\' i}a at FES Acatl{\' a}n, Universidad Nacional Aut{\' o}noma de M{\' e}xico, 53150, State of Mexico, Mexico. E-mail address: {\tt christian.rubio@apolo.acatlan.unam.mx}.}

\begin{abstract} 
In this paper, we study the achromatic and the pseudoachromatic numbers of planar and outerplanar graphs as well as planar graphs of girth 4 and graphs embedded on a surface. We give asymptotically tight results and lower bounds for maximal embedded graphs.
\end{abstract}
\section{Introduction} \label{sec:Intro}
A \emph{$k$-coloring} of a finite and simple graph $G$ is a surjective function that assigns to each vertex of $G$ a color from a set of $k$ colors. A $k$-coloring of a graph $G$ is called \emph{proper} if any two adjacent vertices have different colors, and it is called \emph{complete} if every pair of colors appears on at least one pair of adjacent vertices. 

The \emph{chromatic number $\chi(G)$} of $G$ is the smallest number $k$ for which there exists a proper $k$-coloring of $G$. Note that every proper $\chi(G)$-coloring of a graph $G$ is a complete coloring, otherwise, if two colors $c_i$ and $c_j$ have no edge in common, recolor the vertices colored $c_i$ with $c_j$ obtaining a proper $(\chi(G)-1)$-coloring of $G$ which is a contradiction.

The \emph{achromatic number} $\psi(G)$ of $G$ is the largest number $k$ for which there exists a proper and complete $k$-coloring of $G$. The \emph{pseudoachromatic number} $\psi_s(G)$ of $G$ is the largest number $k$ for which there exists a complete, not necessarily proper, $k$-coloring of $G$, see \cite{MR532949}. From the definitions,
\begin{equation} \label{eq1}
\chi(G)\leq\psi(G)\leq\psi_s(G).
\end{equation}
The origin of achromatic numbers dates back to the 1960s, when Harary, et al. \cite{MR0272662} used them in the studied of homomorphism into the complete graph and Gupta \cite{MR0256930} proved bounds on the sum of the chromatic, achromatic and pseudoachromatic numbers of a graph and its complement. Observe that if a graph $G$ has a complete $k$-coloring, then $G$ must contain at least $\binom{k}{2}$ edges. Consequently, for a graph $G$ with $m$ edges and $k=\psi_s(G)$, $\binom{k}{2}\leq m$ and then (see \cite{MR2450569})
\begin{equation} \label{eq2}
\psi_s(G)\leq \left\lfloor \sqrt{2m+1/4}+1/2\right\rfloor.
\end{equation}
Some results about these parameters can be found in \cite{MR3461960,MR3249588,MR3774452,MR3695270,MR1224703}. 

A graph $G$ is \emph{planar} if $G$ can be drawn, or embedded, in the plane without any two of its edges crossing. Planar graphs and their colorings have been the main topic of intensive research since the beginnings of graph theory because of their connection to the well-known \emph{four color problem} \cite{MR2450569,MR0216979}.
Planar graphs are closely related to outerplanar graphs, namely, a graph $G$ is \emph{outerplanar} if there exists a planar embedding of $G$ so that every vertex of $G$ lies on the boundary of the exterior face of $G$. A planar graph $G$ of order $n\geq3$ has size at most $3(n-2)$, and if its order is $n\geq4$ and its girth is $4$, $G$ has size at most $2(n-2)$ (see \cite{MR2159259}), while an outerplanar graph has size at most $2n-3$. Hence, Equation \ref{eq2} gives the following upper bounds in terms of the number of vertices.
\begin{equation} \label{eq3}
\psi_s(G)\leq \left\lfloor \sqrt{6n - 47/4}+1/2\right\rfloor \textrm{ if }G\textrm{ is planar, }
\end{equation}
\begin{equation} \label{eq4}
\psi_s(G)\leq \left\lfloor \sqrt{4n - 23/4}+1/2\right\rfloor \textrm{ if }G\textrm{ is outerplanar and }
\end{equation}
\begin{equation} \label{eq5}
\psi_s(G)\leq \left\lfloor \sqrt{4n - 31/4}+1/2\right\rfloor \textrm{ if }G\textrm{ is planar of girth at least 4.}
\end{equation}

In this research, we study the relation between the achromatic and pseudoachromatic numbers and planar graphs, outerplanar graphs and planar graphs of girth at least 4, in order to do this, we use the concept of thickness, outerthickness and 4-girth-thickness, respectively. The \emph{thickness} $\theta(G)$ of a graph $G$ is defined as the least number of planar subgraphs whose union is $G$. This parameter was studied in complete graphs, firstly, by Beineke and Harary in \cite{MR0164339,MR0186573} (see also \cite{MR0155314,MR0159318}) and by Alekseev and Gon{\v{c}}akov in \cite{MR0460162}; these authors gave specific constructions of planar graphs to decompose the complete graph $K_n$ using specific matrices. The \emph{outerthickness} $\theta_o(G)$ of $G$ is defined similarly to $\theta(G)$ but with outerplanar instead of planar \cite{MR0347652,MR1100049}. Analogously, the \emph{$4$-girth-thickness} $\theta(4,G)$ of a graph $G$ also is defined similarly to $\theta(G)$ but using planar subgraphs of girth at least $4$ instead of planar subgraphs \cite{R}. Some previous results for complete colorings in planar and outerplanar graphs were obtained by Balogh, et al. in \cite{MR2424828} for the parameter called \emph{the Grundy number}.

This paper is organized as follows. In Section \ref{Planar case}, we use the thickness of the complete graph to prove that the upper bounds of Equation \ref{eq3} is tight for planar graphs. For the sake of completeness, we give a decomposition of $K_{6t+1}$, in a combinatorial way, in order to obtain its thickness. We also prove any maximal planar graph of $n$ vertices has pseudoachromatic number of order $\sqrt{n}$. In summary, we prove the following theorem.
\begin{theorem} \label{teo1}
i) For $n=6t^2+3t+1$ and $t\geq 1$, there exists a planar graph $G$ of order $n$ such that
\[\psi(G)=\psi_s(G)=\left\lfloor \sqrt{6n-47/4}+1/2\right\rfloor.\]
ii) For $n\geq10$, there exists a planar graph $H$ of order $n$ such that
\[\psi(H) \geq \left\lfloor \sqrt{6n-47/4}-9/2\right\rfloor.\]
iii) For some integer $n_0$ and any maximal planar graph $G$ of order $n\geq n_0$, \[0.81\sqrt{n}\leq\psi_s(G)\leq 2.45\sqrt{n}.\]
\end{theorem}
While Theorem \ref{teo1} is about planar graphs, in Section \ref{Outerplanar case}, we prove the following theorem about outerplanar graphs.
\begin{theorem} \label{teo2}
i) For $n=4t^2+1$ and $t\geq 1$, there exists an outerplanar graph $G$ of order $n$ such that
\[\psi(G)=\psi_s(G)=\left\lfloor \sqrt{4n-23/4}+1/2\right\rfloor.\]
ii) For $n\geq5$, there exists an outerplanar graph $H$ of order $n$ such that
\[\psi(H) \geq  \left\lfloor \sqrt{4n-23/4}-5/2\right\rfloor.\]
iii) For some integer $n_0$ and any maximal outerplanar graph $G$ of order $n\geq n_0$, \[1.41\sqrt{n}\leq\psi_s(G)< 2\sqrt{n}+1/2.\]
\end{theorem}
We remark that in \cite{hara2002achromatic}, the authors prove a related result to \emph{iii)} of Theorem \ref{teo2}, namely, $\psi(G)\geq \sqrt{n}-1$ for any maximal outerplanar graph of order $n\geq 4$ and $\psi(G)\geq 1.41\sqrt{n}-2$ for any maximal outerplanar graph of order $n\geq 4$ and with two vertices of degree 2.

Next, in Section \ref{4 girth case}, we prove the following theorem about planar graphs of girth at least 4.
\begin{theorem} \label{teo3}
i) For $n=4t^2+2$ and $t\geq 1$, there exists a maximal planar graph $G$ of order $n$ and girth at least $4$ such that
\[\psi(G)=\psi_s(G)=\left\lfloor\sqrt{4n-31/4}+1/2\right\rfloor.\]
ii) For $n\geq6$, there exists a planar graph $H$ of order $n$ and girth at least $4$ such that
\[\psi(H) \geq  \left\lfloor \sqrt{4n-31/4}-5/2\right\rfloor.\]
iii) For any $n\geq 4$, there exists a maximal planar graph $G$ of order $n$ with girth $4$ such that, \[\psi(G)=2\qquad and \qquad \psi_s(G)=3.\]
\end{theorem}
In Section \ref{Other surfaces}, we give results for graphs embedded on an surface $S$ in a similar way. Finally, in Section \ref{The Platonic graphs}, we give the achromatic numbers of the Platonic graphs.
\section{Planar graphs and the thickness of $K_{6t+1}$} \label{Planar case}
In this section, we exhibit a planar graph with the property that its achromatic numbers attain the upper bound of Equation \ref{eq3}. To achieve it, we construct an almost maximal planar graph and color its $n$ vertices using essentially $\sqrt{6n}$ colors in a proper way such that there is exactly an edge between each pair of colors. Next, we give a general lower bound for the pseudoachromatic number of maximal planar graphs.

To begin with, we show that the thickness of $K_{6t}$ and $K_{6t+1}$ is $t+1$, that is, both graphs are the union of $t+1$ planar subgraphs. Such a decomposition originally appears in \cite{MR0164339}, for a further explanation see \cite{MR0460162,MR0186573,MR2626173}. In the following subsection, we give a pure combinatorial approach to that decomposition.
\subsection{The planar decomposition of $K_{6t+1}$} \label{Planar decomposition}
A well-known result about complete graphs of even order $2t$ is that these graphs are decomposable into a cyclic factorization of Hamiltonian paths, see \cite{MR2450569}. In the remainder of this section, all sums are taken modulo $2t$.

Let $G^x$ be a complete graph of order $2t$. Label its vertex set $V(G^x)$ as $\{x_1,x_2,\dots,x_{2t}\}$.
Let $\mathcal{F}^x_i$ be the Hamiltonian path with edges \[x_{i}x_{i+1},x_{i+1}x_{i-1},x_{i-1}x_{i+2},x_{i+2}x_{i-2},\dots,x_{i+t+1}x_{i+t},\] where $i\in\{1,2,\dots,t\}$. Such factorization of $G^x$ is the partition $\{E(\mathcal{F}^x_1),E(\mathcal{F}^x_2),\dots,E(\mathcal{F}^x_t)\}$. Note that the center of $\mathcal{F}^x_i$ has the edge $e^x_{i+\left\lceil \frac{t}{2}\right\rceil}=x_{i+\left\lceil \frac{t}{2}\right\rceil}x_{i+\left\lceil \frac{t}{2}\right\rceil+t}$, see Figure \ref{Fig1}.
\begin{figure}[!htbp]
\begin{center}
\includegraphics{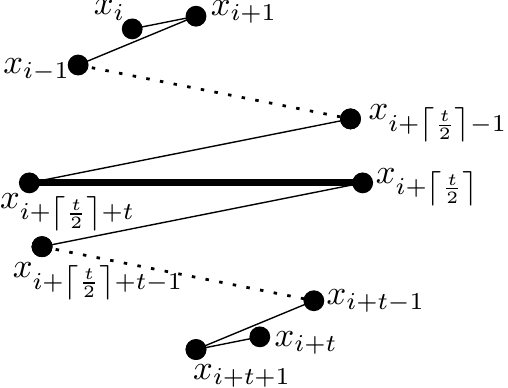}
\caption{The Hamiltonian path $\mathcal{F}^x_i$ with the edge $e^x_{i+\left\lceil \frac{t}{2}\right\rceil}$ in bold.}\label{Fig1}
\end{center}
\end{figure}

Consider the complete subgraphs $G^u$, $G^v$ and $G^w$ of $K_{6t}$ such that each of them has $2t$ vertices, $G^v$ is a subgraph of $K_{6t}\setminus V(G^u)$ and $G^w$ is $K_{6t}\setminus (V(G^u)\cup V(G^v))$. Label their vertex sets $V(G^u)$, $V(G^v)$ and $V(G^w)$ as $\{u_1,u_2,\dots,u_{2t}\}$, $\{v_1,v_2,\dots,v_{2t}\}$ and $\{w_1,w_2,\dots,w_{2t}\}$, respectively. 

Next, for any symbol $x$ of $\{u,v,w\}$, we consider the cyclic factorization $\{E(\mathcal{F}^x_1),E(\mathcal{F}^x_2),\dots,E(\mathcal{F}^x_t)\}$ of $G^x$ into Hamiltonian paths and we denote as $P_{x_i}$ and $P_{x_{i+t}}$ the $t$-subpaths of $\mathcal{F}^x_i$ containing the leaves $x_i$ and $x_{i+t}$, respectively.

Now, we construct the maximal planar subgraphs $G_1$, $G_2$,...,$G_{t}$ and a matching $G_{t+1}$, each of them of order $6t$, as follows. 

Let $G_{t+1}$ be a perfect matching with edges $u_ju_{j+t}$, $v_jv_{j+t}$ and $w_jw_{j+t}$ for $j\in\{1,2,\dots,t\}$. 

For $i\in\{1,2,\ldots,t\}$, let $G_i$ be the spanning planar graph of $K_{6t}$ whose adjacencies are given as follow:
we take the 6 paths, $P_{u_i},P_{u_{i+t}},P_{v_i},P_{v_{i+t}},P_{w_i}$ and $P_{w_{i+t}}$ and insert them in the octahedron with vertices $u_{i},u_{i+t},v_{i},v_{i+t},w_{i}$ and $w_{i+t}$ as is shown in Figure \ref{Fig2} (Left). The vertex $x$ of each path $P_x$ is identified with the vertex $x$ in the corresponding triangle face and join all the other vertices of the path with both of the other vertices of the triangle face, see Figure \ref{Fig2} (Right).
\begin{figure}[!htbp]
\begin{center}
\includegraphics{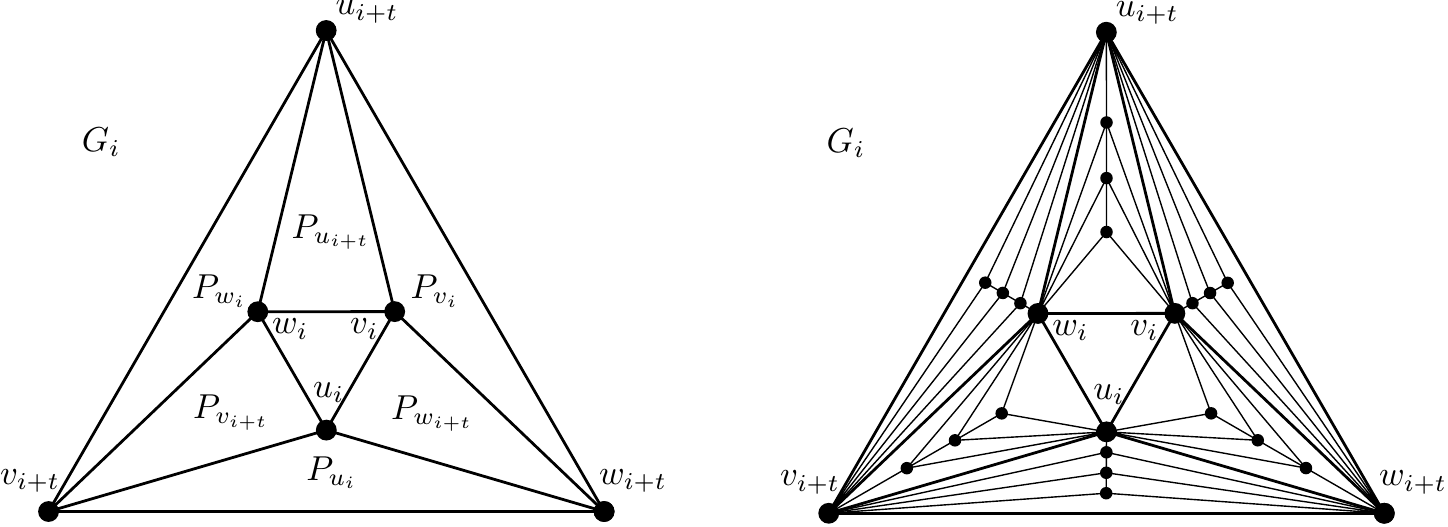}
\caption{(Left) The octahedron diagram of the graph $G_{i}$. (Right) The graph $G_{i}$.}\label{Fig2}
\end{center}
\end{figure}

Given the construction of $G_i$, $K_{6t}=\overset{t+1}{\underset{i=1}{\bigcup}}G_{i}$, see \cite{MR0460162,MR0164339} for details
. Therefore, the resulting $t+1$ planar subgraphs show that $\theta(K_{6t})\leq t+1$. Hence, $\theta(K_{6t})= t+1$ due to the fact that $\theta(K_{6t})\geq\left\lceil \frac{\binom{6t}{2}}{3(6t-2)}\right\rceil =t+1.$

The case of $K_{6t+1}$ is based on this decomposition since $K_{6t+1}$ is the join of $K_{6t}$ and a new vertex $z$. Hence, take the set of subgraphs $\{G_1,G_2,\dots,G^*_{t+1}\}$, where $G^*_{t+1}$ is the join between the vertex $z$ and the matching $G_{t+1}$ (see Figure \ref{Fig3} for a drawing of $G^*_{t+1}$), to obtain a planar decomposition of $K_{6t+1}$ using $t+1$ elements, then $\theta(K_{6t+1})\leq t+1$. Therefore $\theta(K_{6t+1})= t+1$ due to the fact that $\theta(K_{6t+1})\geq\left\lceil \frac{\binom{6t+1}{2}}{3(6t-1)}\right\rceil =t+1$.
\subsection{On Theorem \ref{teo1}}\label{On Theorem 1}
In this subsection, we show the existence of a planar graph $G$ of order $n=6t^2+3t+1$ and color its vertices in a proper and complete way using $k=\lfloor 1/2+\sqrt{6n-47/4}\rfloor$ colors in order to prove Theorem \ref{teo1} i). Next, we prove Theorem \ref{teo1} ii), i.e., the upper bound given in Equation \ref{eq3} for planar graphs is asymptotically the best possible. Finally, we prove Theorem \ref{teo1} iii), the natural question about a lower bound for the pseudoachromatic number of maximal planar graphs.
\begin{proof}[Proof of Theorem \ref{teo1}]
i)
Consider the set of planar subgraphs $\{G_1,G_2,\dots,G^*_{t+1}\}$ of $K_{6t+1}$ described in Subsection \ref{Planar decomposition} but now, color each vertex with their corresponding label. Take a planar drawing of $G_i$ such that its colored vertices $u_{i+t}$, $v_{i+t}$ and $w_{i+t}$ are the vertices of the exterior face, for $i\in\{1,2,\ldots,t\}$. Next, take a planar drawing of $G^*_{t+1}$ (lying in the exterior face) identifying its colored vertices $u_{i+t}$, $v_{i+t}$ and $w_{i+t}$ with the colored vertices $u_{i+t}$, $v_{i+t}$ and $w_{i+t}$ of the exterior face, see Figure \ref{Fig3}. The resulting planar graph $G$ has a proper and complete coloring with $6t+1$ colors and order $n=t(6t)+(6t)/2+1=6t^2+3t+1$.

Indeed, the coloring is proper and complete due to fact that every pair of different colors $x_j$ and $x_{j'}$ are the labels of the complete graph $K_{6t+1}$, the edge $x_jx_{j'}$ is an edge in some subplanar graph $G_1$, $G_2$,...,$G_t$, $G^*_{t+1}$ exactly once and then, there exists an edge of $G$ with colors $x_j$ and $x_{j'}$. Therefore $6t+1\leq \psi(G)$.  We call this colored graph $G$ as \emph{the optimal colored planar graph of $n=6t^2+3t+1$ vertices}.
\begin{figure}[!htbp]
\begin{center}
\includegraphics{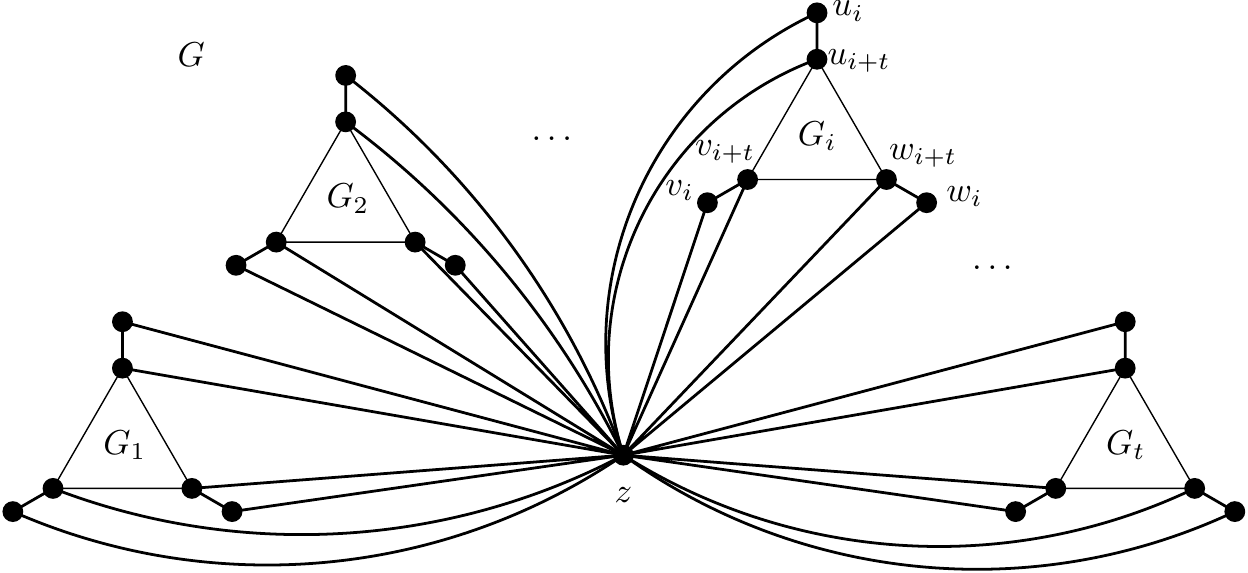}
\caption{The planar graph $G$ with the subgraph $G^*_{t+1}$ in bold.}\label{Fig3}
\end{center}
\end{figure}

On the other hand, for $n=6t^2+3t+1$, $\left\lfloor\sqrt{6n-47/4}+1/2\right \rfloor$ is equal to \[\left\lfloor\sqrt{36t^2+18t-23/4}+1/2\right \rfloor.\]
Since $\sqrt{36t^2+18t-23/4}+1/2<6t+2$ because $36t^2+18t-23/4\leq(6t+3/2)^2=36t^2+18t+9/4$ and $6t+1\leq \sqrt{36t^2+18t-23/4}+1/2$ because $36t^2+6t+1/4=(6t+1/2)^2\leq36t^2+18t-23/4,$ it follows that \[\left\lfloor\sqrt{36t^2+18t-23/4}+1/2\right \rfloor=6t+1.\]

Finally, we have \[\left\lfloor \sqrt{6n-47/4}+1/2\right\rfloor\leq\psi(G)\leq\psi_s(G)\leq\left\lfloor \sqrt{6n-47/4}+1/2\right\rfloor\]and the result follows.

ii)
Assume that $6t^2+3t+1\leq n<N=6(t+1)^2+3(t+1)+1$ for some natural number $t$. Let $G$ be the planar graph of order $6t^2+3t+1$ constructed in the proof of Theorem \ref{teo1} i); and let $H$ be the graph of order $n$ obtained by adding $n-(6t^2+3t+1)$ vertices to $G$. It is clearly that $H$ is a planar graph of order $n$ with achromatic number $6t+1$. Since $n<N$, we have that \[\left\lfloor\sqrt{6n-47/4}+1/2\right \rfloor<\left\lfloor\sqrt{6N-47/4}+1/2\right \rfloor=6(t+1)+1=\psi(H)+6.\]
Therefore \[\psi(H) \geq \left\lfloor\sqrt{6n-47/4}+1/2\right \rfloor-5 \]
and the result follows.

iii)
If a maximal planar graph $G$ of order $n$ has a Hamiltonian cycle, for instance if $G$ is $4$-connected \cite{MR0081471}, then its pseudoachromatic number $\psi_s(G)$ is at least $\psi_s(C_n)$ which is equal to $\max\{k:k \left\lfloor k/2 \right \rfloor \leq n\}\sim \sqrt{2n}$, see \cite{MR2450569,MR0441778,MR1811221}. In consequence, for some integer $n_0$, if $G$ is a Hamiltonian maximal planar graph of order $n\geq n_0$ then \[1.41\sqrt{n}\leq\psi_s(G).\]
On the other hand, since every maximal planar graph $G$ contains a matching of size at least $\frac{n+4}{3}$ \cite{MR548625} then its pseudoachromatic number is at least $\max\{k:\binom{k}{2} \leq \frac{n+4}{3}\}\sim \sqrt{2n/3}$. Therefore, for some integer $n_0$, if $G$ is a maximal planar graph of order $n\geq n_0$ then \[0.81\sqrt{n}\leq\psi_s(G)\]
and the result follows.
\end{proof}
\section{Outerplanar graphs and the outerthickness of $K_{4t}$} \label{Outerplanar case}

In this section, we prove Theorem \ref{teo2} and we exhibit a decomposition of $K_{4t}$ into $t+1$ outerplanar subgraphs, with a slightly different labelling than the original decomposition published in \cite{MR1100049}.
 The proof of this theorem follows the same technique as previous one.
\subsection{The outerplanar decomposition of $K_{4t}$} \label{Outerplanar decomposition}
In order to show the outerplanar decomposition of $K_{4t}$, we use the labelling called \emph{boustrophedon} of a path $P_{u_{i+t}}$ for $i\in\{1,2,\dots,t\}$, see Figure \ref{Fig4} (Left) and \cite{MR1100049}. In the remainder of this section, all sums are taken modulo $4t$.

Label the vertex set $V(K_{4t})$ as $\{u_1,u_2,\dots,u_{4t}\}$. First construct the path $P_{u_{i+t}}$ and the vertex $u_{i+4t}$ is joined to each of the $r$ vertices of $P_{u_{i+t}}$. Then, take three copies of the obtained graph, with the labels increased by $t$, $2t$ and $3t$. On identifying vertices having the same label $i+t$, $i+2t$, $i+3t$ and $i+4t$ and adding the edge $u_{i+t}u_{i+3t}$ to obtain maximal outerplanar subgraphs $G_i$, each of them of order $4t$ for $i\in\{1,2,\dots,t\}$, and a matching $G_{t+1}$ of order $2t$ with the edges $u_{i}u_{i+2t}$ for $i\in\{1,2,\dots,t\}$.
\begin{figure}[!htbp]
\begin{center}
\includegraphics{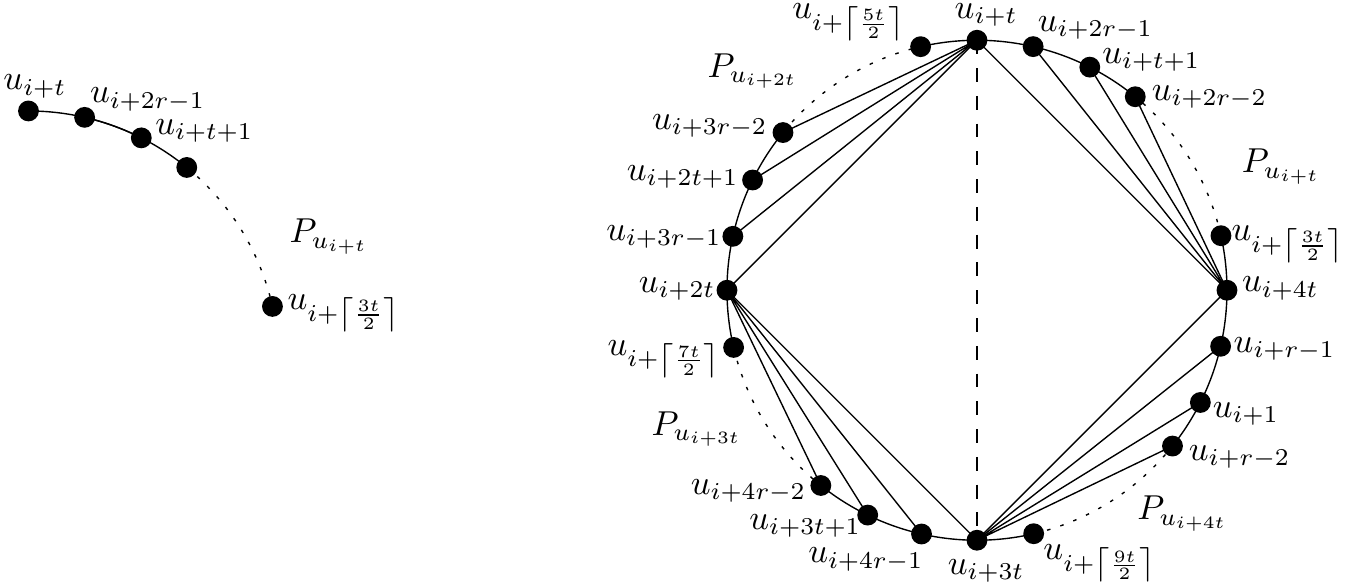}
\caption{(Left) The boustrophedon labelling. (Right) The graph $G_{i}$.}\label{Fig4}
\end{center}
\end{figure}

The resulting $t+1$ outerplanar subgraphs show that $\theta_o(K_{4t})\leq t+1$. Hence $\theta_o(K_{4t})= t+1$ owing to the fact that $\theta_o(K_{4t})\geq\left\lceil \frac{\binom{4t}{2}}{2(4t)-3}\right\rceil =t+1$.
\subsection{On Theorem \ref{teo2}}\label{On Theorem 2}
The proof of Theorem \ref{teo2} is given in this subsection. To prove i), we exhibit an outerplanar graph $G$ of order $n=4t^2+1$ and color its vertices in a proper and complete way using $k=\lfloor 1/2+\sqrt{4n-23/4}\rfloor$ colors. To prove ii), we show that for an arbitrary $n$, the upper bound of Equation \ref{eq4} is sharp for a given outerplanar graph of order $n$. To prove iii), we show a lower bound for the pseudoachromatic number of maximal outerplanar graphs.
\begin{proof}[Proof of Theorem \ref{teo2}]
i)
Consider the set of planar subgraphs $\{G_1,G_2,\dots,G_{t+1}\}$ of $K_{4t}$ described in Subsection \ref{Outerplanar decomposition} and color each vertex with the corresponding label of the decomposition. For $i\in\{1,2,\ldots,t\}$, take an outerplanar drawing of $G_i$ such that its vertices are in the exterior face. Next, take a planar drawing of $G_{t+1}$ (lying in the exterior face) identifying its colored vertices $u_{i}$ with the colored vertices $u_{i}$ of $G_{i}$, for $i\in\{1,2,\ldots,t\}$, and identifying its colored vertices $u_{i+2t}$ with the colored vertices $u_{i+2t}$ of $G_{i+1}$, for $i\in\{1,2,\ldots,t-1\}$, see Figure \ref{Fig5}. The resulting planar graph $G$ has a proper and complete coloring with $4t$ colors and order $n=t(4t)+1=4t^2+1$. 

Indeed, the coloring is proper and complete because every pair of different colors $x_j$ and $x_{j'}$ are the labels of the complete graph $K_{4t}$, the edge $x_jx_{j'}$ is an edge in some subplanar graph $G_1$, $G_2$,...,$G_t$, $G_{t+1}$ exactly once and then, there exists an edge of $G$ with colors $x_j$ and $x_{j'}$. Therefore $4t\leq \psi(G)$.
\begin{figure}[!htbp]
\begin{center}
\includegraphics{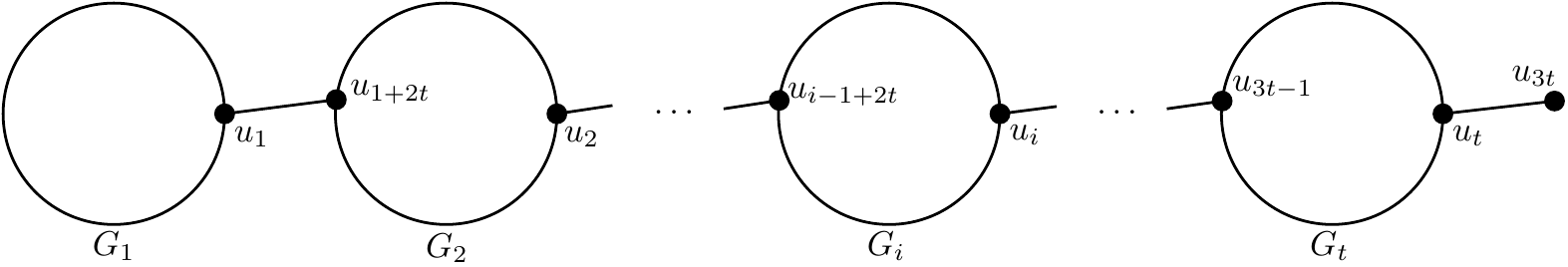}
\caption{The outerplanar graph $G$.}\label{Fig5}
\end{center}
\end{figure}

On the other hand, for $n=4t^2+1$, $\left\lfloor\sqrt{4n-23/4}+1/2\right \rfloor$ is equal to $\left\lfloor\sqrt{16t^2-7/4}+1/2\right \rfloor.$
Since $\sqrt{16t^2-7/4}+1/2<4t+1$ because $16t^2-7/4\leq(4t+1/2)^2=16t^2+4t+1/4$ and $4t\leq \sqrt{16t^2-7/4}+1/2$ because $16t^2-8t+1/4=(4t-1/2)^2\leq 16t^2-7/4,$ it follows that \[\left\lfloor\sqrt{16t^2-7/4}+1/2\right \rfloor=4t.\]

Finally, we have \[\left\lfloor \sqrt{4n-23/4}+1/2\right\rfloor\leq\psi(G)\leq\psi_s(G)\leq\left\lfloor \sqrt{4n-23/4}+1/2\right\rfloor\]and the result follows.

ii)
Let $n$ and $t$ be natural numbers such that $4t^2+1\leq n<N=4(t+1)^2+1$, and let $G$ be the outerplanar graph of order $4t^2+1$ constructed in the proof of Theorem \ref{teo2} i). Let $H$ be the graph of order $n$ obtained by adding $n-(4t^2+1)$ vertices to the exterior face of $G$. Therefore $H$ is an outerplanar graph of order $n$ with achromatic number $4t$. Now, since $n<N$ then \[\left\lfloor\sqrt{4n-23/4}+1/2\right \rfloor<\left\lfloor\sqrt{4N-23/4}+1/2\right \rfloor=4(t+1)=\psi(H)+4.\]
Therefore \[\psi(H) \geq \left\lfloor\sqrt{4n-23/4}+1/2\right \rfloor-3\]
and the result follows.

iii)
Since a maximal outerplanar graph $G$ of order $n$ is Hamiltonian, its pseudoachromatic number $\psi_s(G)$ is at least $\psi_s(C_n)=\max\{k:k \left\lfloor k/2 \right \rfloor \leq n\}\sim \sqrt{2n}$, see \cite{MR2450569,MR0441778,MR1811221}. In conclusion, for some integer $n_0$, if $G$ is a maximal outerplanar graph of order $n\geq n_0$ then \[1.41\sqrt{n}\leq\psi_s(G)\]
and the result follows.
\end{proof}
\section{Planar graphs of girth $4$ and the $4$-girth-thickness of $K_{4t}$} \label{4 girth case}
In this section, we prove Theorem \ref{teo3} and we exhibit a planar graph of girth $4$ with achromatic number the highest possible, after that, we use such graph to prove that the rather upper bound of Equation \ref{eq5} is asymptotically correct, and to end, we show that the triangle-free condition is enough to obtain $n$-graphs with constant achromatic numbers for arbitrary $n$. Firstly, we show that the 4-girth-thickness of $K_{4t}$ is $t+1$, see \cite{R}.
\subsection{The $4$-girth-planar decomposition of $K_{4t}$} \label{4-girth-planar decomposition}
In order to show the $4$-girth-planar decomposition of $K_{4t}$, we also use the cyclic factorization into Hamiltonian paths of $K_{2t}$ presented in Section \ref{Planar decomposition}. In the remainder of this section, all sums are taken modulo $2t$.

Let $G^u$ and $G^v$ be the complete subgraphs of $K_{4t}$ isomorphic to $K_{2t}$ such that $G^v$ is the subgraph $K_{4t}\setminus V(G^u)$. Label their vertex sets $V(G^u)$ and $V(G^v)$ as $\{u_1,u_2,\dots,u_{2t}\}$ and $\{v_1,v_2,\dots,v_{2t}\}$, respectively.

For any symbol $x$ of $\{u,v\}$, we consider the cyclic factorizations $\{E(\mathcal{F}^x_1),E(\mathcal{F}^x_2),\dots,E(\mathcal{F}^x_t)\}$. 

Then we construct the planar subgraphs $G_1$, $G_2$,...,$G_t$ of girth $4$, order $4t$ and size $8t-4$ (observe that $2(4t-2)=8t-4$), and also the perfect matching $G_{t+1}$ of order $4t$, as follows. Let $G_i$ be a spanning subgraph of $K_{4t}$ with edges $E(\mathcal{F}^u_i)\cup E(\mathcal{F}^v_i)$ and \[u_{i}v_{i+1},v_{i}u_{i+1},u_{i+1}v_{i-1},v_{i+1}u_{+i-1},u_{i-1}v_{i+2},v_{i-1}u_{i+2},\dots,u_{i+t+1}v_{i+t},v_{i+t+1}u_{i+t}\] where $i\in\{1,2,\dots,t\}$; and let $G_{t+1}$ be the matching with edges $u_jv_j$ for $j\in\{1,2,\dots,2t\}$. Figure \ref{Fig6} shows that $G_i$ is a planar graph of girth at least $4$. Note that each planar graph $G_i$ has a drawing with the vertices $u_i$, $u_{i+1}$, $v_i$ and $v_{i+1}$ in the exterior face. 
\begin{figure}[!htbp]
\begin{center}
\includegraphics{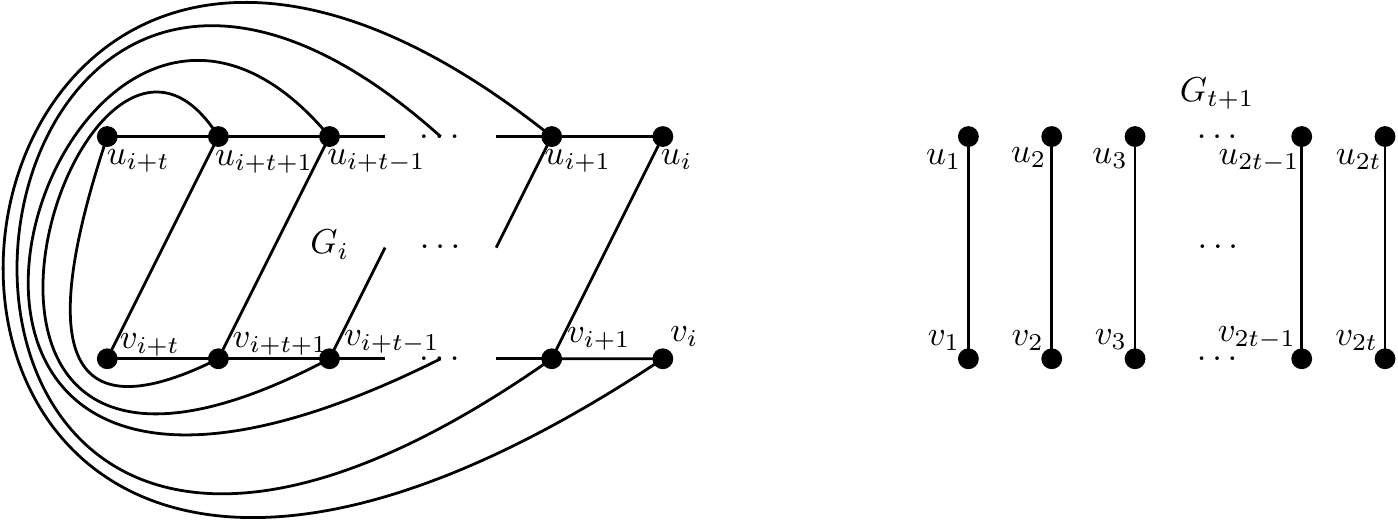}
\caption{\label{Fig6} (Left) The graph $G_i$. (Right) The graph $G_{t+1}$.}
\end{center}
\end{figure}

The resulting $t+1$ planar subgraphs of girth at least $4$ show that $\theta(4,K_{4t})\leq t+1$. Hence $\theta(4,K_{4t})= t+1$ because $\theta(4,K_{4t})\geq\left\lceil \frac{\binom{4t}{2}}{2(4t-2)}\right\rceil =t+1$, see \cite{R}.
\subsection{On Theorem \ref{teo3}}\label{On Theorem 3}
In this subsection is proven Theorem \ref{teo3}. To begin with, we show a planar graph $G$ of order $n=4t^2+2$ and girth $4$, and we color its vertices in a proper and complete way using $k=\lfloor 1/2+\sqrt{4n-31/4}\rfloor$ colors. Next, we show that for any $n$, the upper bound of Equation \ref{eq5} is sharp for a given planar graph of order $n$ and girth $4$. Finally, we show that for any maximal planar graph $G$ of girth $4$ has a tight constant lower bound.
\begin{proof}[Proof of Theorem \ref{teo3}]
i)
Consider the set of planar subgraphs $\{G_1,G_2,\dots,G_{t+1}\}$ of $K_{4t}$ described in Subsection \ref{4-girth-planar decomposition} and color each vertex with the corresponding label of the decomposition. For $i\in\{1,\ldots,t\}$, take a planar drawing of $G_i$ such that the vertices $u_i$, $u_{i+1}$, $v_i$ and $v_{i+1}$ are in the exterior face, see Figure \ref{Fig6} (Left).

Insert $G_t$ in the adjacent face to the exterior face of $G_{t-1}$ with vertices $u_{t-1}$, $u_{t}$, $v_{t-2}$ and $v_{t}$ and identify the colored vertices $u_{t}$ and $v_t$ of $G_t$ with the colored vertices $u_{t}$ and $v_t$ of $G_{t-1}$, respectively. We call to the resulting graph $G'_{t-1}$, see Figure \ref{Fig7} (Left). Insert $G'_{t-1}$ in the adjacent face to the exterior face of $G_{t-2}$ with vertices $u_{t-1}$, $u_{t-3}$, $v_{t-1}$ and $v_{t-2}$ and identify the colored vertices $u_{t-1}$ and $v_{t-1}$ of $G'_{t-1}$ with the colored vertices $u_{t-1}$ and $v_{t-1}$ of $G_{t-2}$, respectively. We call to the resulting graph $G'_{t-2}$. We repeat the procedure to obtain the graph $G'_1$.

Finally, insert every edge $u_jv_j$ of $G_{t+1}$, for $j\in\{1,2,\dots,2t\}$, in some face with a vertex of $G'_1$ colored $u_j$ and identify them, see Figure \ref{Fig7} (Right). The resulting planar graph $G$ has a proper and complete coloring with $4t$ colors and order $n=4t+(t-1)(4t-2)+2t=4t^2+2$. 

Indeed, the coloring is proper and complete because every pair of different colors $x_j$ and $x_{j'}$ are the labels of the complete graph $K_{4t}$, the edge $x_jx_{j'}$ is an edge in some subplanar graph $G_1$, $G_2$,...,$G_t$, $G_{t+1}$ exactly once and then, there exists an edge of $G$ with colors $x_j$ and $x_{j'}$. Therefore $4t\leq \psi(G)$.
\begin{figure}[!htbp]
\begin{center}
\includegraphics{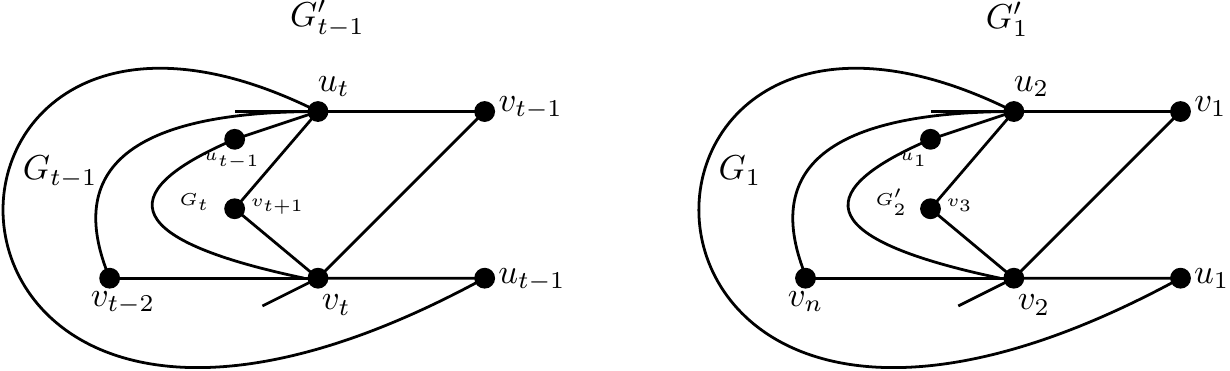}
\caption{(Left) The planar graph $G'_{t-1}$. (Right) The planar graph $G'_1$.}\label{Fig7}
\end{center}
\end{figure}

On the other hand, for $n=4t^2+2$, $\left\lfloor\sqrt{4n-31/4}+1/2\right \rfloor$ is equal to $\left\lfloor\sqrt{16t^2-1/4}+1/2\right \rfloor.$
Since $\sqrt{16t^2-1/4}+1/2<4t+1$ since $16t^2-1/4\leq(4t+1/2)^2=16t^2+4t+1/4$ and $4t\leq \sqrt{16t^2-1/4}+1/2$ since $16t^2-8t+1/4=(4t-1/2)^2\leq 16t^2-1/4,$ it follows that \[\left\lfloor\sqrt{16t^2-1/4}+1/2\right \rfloor=4t.\]

In conclusion, we have \[\left\lfloor \sqrt{4n-31/4}+1/2\right\rfloor\leq\psi(G)\leq\psi_s(G)\leq\left\lfloor \sqrt{4n-31/4}+1/2\right\rfloor\]and the result follows.

ii)
Let $n$ and $t$ be natural numbers such that $4t^2+2\leq n<N=4(t+1)^2+2$, and let $G$ be the planar graph of order $4t^2+2$ and girth $4$ constructed in the proof of Theorem \ref{teo3} i). Let $H$ be the graph of order $n$ obtained by adding $n-(4t^2+2)$ vertices to $G$. Therefore $H$ is a planar graph of order $n$ with girth $4$ and achromatic number $4t$. Now, since $n<N$ then \[\left\lfloor\sqrt{4n-31/4}+1/2\right \rfloor<\left\lfloor\sqrt{4N-31/4}+1/2\right \rfloor=4(t+1)=\psi(H)+4.\]
Therefore \[\psi(H) \geq \left\lfloor\sqrt{4n-31/4}+1/2\right \rfloor-3\]
and the result follows.

iii) Consider the bipartite graph $K_{2,n-2}$ of order $n\geq 4$. This graph is planar of girth $4$ with $2(n-2)$ edges, therefore it is a maximal planar graph of girth $4$. Its achromatic number is $2$, see \cite{MR2450569}. And its pseudoachromatic number is $3$, otherwise, there exist at least two chromatic classes contained in the partition with $n-2$ vertices and then, the coloring is not complete. Since $C_4$ is a subgraph of $K_{2,n-2}$ and it has pseudoachromatic number is $3$, the result follows.
\end{proof}
\section{Graphs embedded on a surface} \label{Other surfaces}
In this section, we present results for surfaces, specifically, we extend Theorem \ref{teo1} to a surface $S$ instead of the plane.

To begin with, we establish some definitions. A \emph{surface} is a topological space $S$, which is compact arc-connected and Hausdorff, such that every point has a neighbourhood homeomorphic to the Euclidean plane $\mathbb{R}^2$. 

An embedding $\sigma\colon G\rightarrow S$ of a graph $G$ in $S$ maps the vertices of $G$ to distinct points in $S$ and each edge $xy$ of $G$ to an arc $\sigma(x)-\sigma(y)$ in $S$, such that no inner point of such an arc is the image of a vertex or lies on another arc. A face of $G$ in $S$ is a component (arc-connected component) of $S\setminus \sigma(G)$ where $\sigma(G)$ is the union of all those points and arcs of $S$, and the subgraph of $G$ that $\sigma$ maps to the frontier of this face is its boundary. When each face is (homeomorphic to) a disc and $K_3$ is its boundary, we say that $G$ is a \emph{maximal $S$-graph}.

For every surface $S$ there is an integer called the \emph{Euler characteristic} such that whenever a graph $G$ with $n$ vertices and $m$ edges is embedded in $S$ so that there are $f$ faces and every face is a disc, then it is equal to $n-m+f$.

We work instead with the invariant $\varepsilon(S)$ defined as $2-n+m-f$ and called the \emph{Euler genus} of $S$, see \cite{MR2159259}. The well-known Classification Theorem of Surfaces says that, up to homeomorphism, every surface is a sphere with some finite number of handles or crosscaps.
\begin{lemma}[See \cite{MR2159259}]
Adding a handle to a surface raises its Euler genus by 2. And adding a crosscap to a surface raises its Euler genus by 1.
\end{lemma}
\begin{lemma}[See \cite{MR1699257}]
A crosshandle is homeomorphic to two crosscaps.
\end{lemma}
\begin{lemma}[Dyck's Theorem. See \cite{MR1699257}]
Handles and crosshandles are equivalent in the presence of a crosscap.
\end{lemma}

For any maximal $S$-graph $G$, $3f=2m$, then $m=3n+3\varepsilon(S)-6$ and, by Equation \ref{eq2} \[\psi_s(G)\leq \left\lfloor \sqrt{6(n+\varepsilon(S))-47/4}+1/2\right\rfloor.\]
We prove the following theorems.
\begin{theorem} \label{teo4}
Let $S$ be an orientable surface with $h\geq0$ handles.

i) For $n=6t^2+3t+1-\varepsilon(S)$, with $n,t\geq 1$, there exists a graph $G$ of order $n$, embeddable in the surface $S$, such that 
\[\psi(G)=\psi_s(G)=\left\lfloor \sqrt{6(n+\varepsilon(S))-47/4}+1/2\right\rfloor.\]
ii) For $n\geq10-\varepsilon(S)\geq 1$, there exists an embeddable $S$-graph $H$ of order $n$ such that
\[\psi(H) \geq  \left\lfloor \sqrt{6(n+\varepsilon(S))-47/4}-9/2\right\rfloor.\]
\end{theorem}
\begin{proof}
i) Consider the colored optimal planar graph $G$ of $n=6t^2+3t+1$ vertices embedded in the plane which was described in the proof of Theorem \ref{teo1}, see Figure \ref{Fig3}. We proceed to add $h$ handles to the plane in the following way.

Cutting along a circle $C_i$ in the triangle face $u_iv_iw_i$ contained in the subgraph $G_i$ of $G$ and also along a circle $C'_i$ in the face containing $z$, which was the exterior face, for all $i\leq \min\{h,t\}$. Now, add a handle connecting both circles. For each $i$, we identify pairs of vertices with the same color, namely, the vertices $u_i$ and $v_i$ of the exterior face with the colored vertices $u_i$ and $v_i$ of the triangle face of $G_i$, respectively, which is possible moving them through the handle, see Figure \ref{Fig8}.
\begin{figure}[!htbp]
\begin{center}
\includegraphics{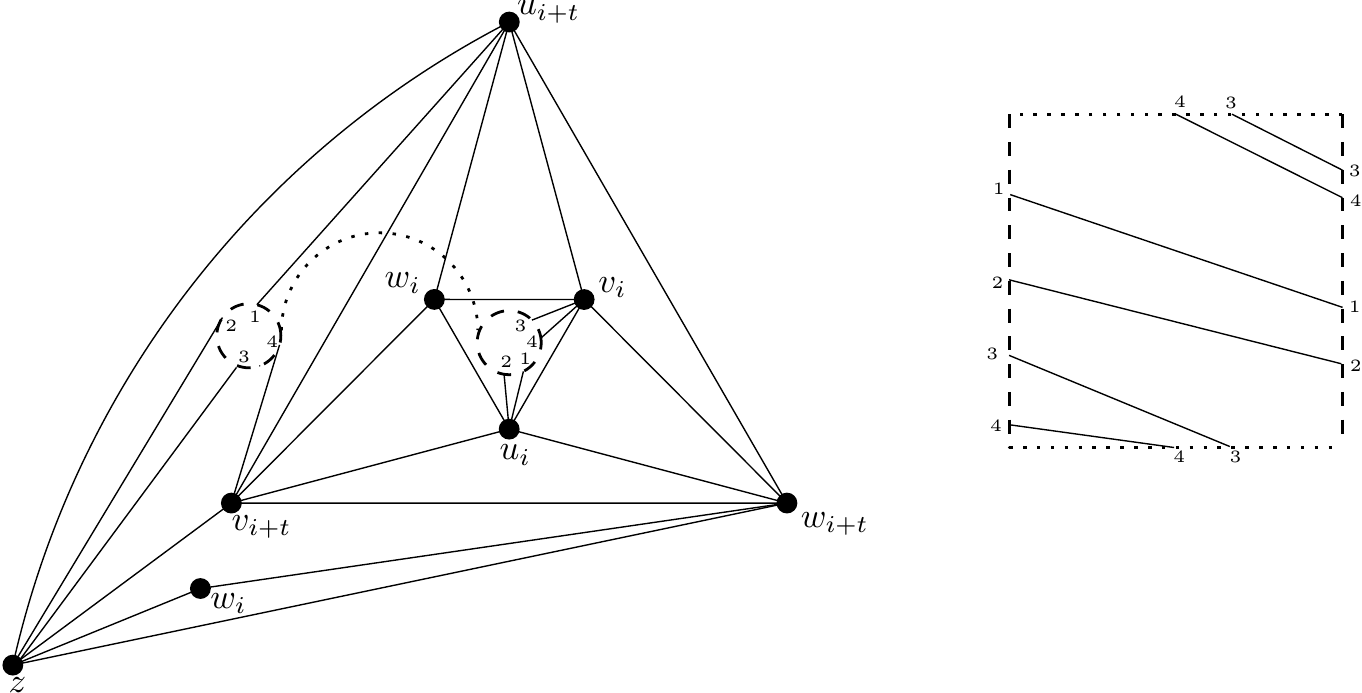}
\caption{Adding a handle and identifying colored vertices $u_i$ and $v_i$.}\label{Fig8}
\end{center}
\end{figure}

We obtain a surface with $h$ handles in which is embedded the graph $G$ of $n=6t^2+3t+1-\varepsilon(S)$ vertices and colored with $6t+1$ colors, see Figure \ref{Fig9}.
\begin{figure}[!htbp]
\begin{center}
\includegraphics{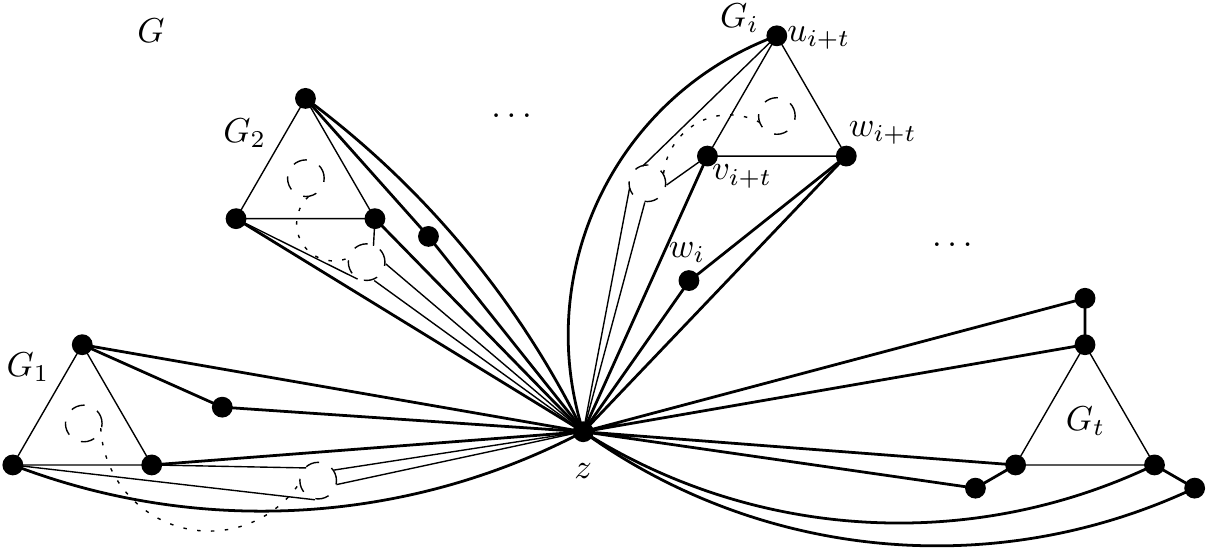}
\caption{The embedded graph $G$ for $i<t$.}\label{Fig9}
\end{center}
\end{figure}

On the other hand, for $n=6t^2+3t+1-\varepsilon(S)$, $\left\lfloor\sqrt{6(n+\varepsilon(S))-47/4}+1/2\right \rfloor$ is equal to  \[\left\lfloor\sqrt{36t^2+18t-23/4}+1/2\right \rfloor=6t+1,\]
we have \[\left\lfloor \sqrt{6(n+\varepsilon(S))-47/4}+1/2\right\rfloor\leq\psi(G)\leq\psi_s(G)\leq\left\lfloor \sqrt{6(n+\varepsilon(S))-47/4}+1/2\right\rfloor\]and the result follows.

ii)
Assume that $6t^2+3t+1-\varepsilon(G)\leq n<N=6(t+1)^2+3(t+1)+1-\varepsilon(G)$ for some natural number $t$. Let $G$ be the $S$-graph of order $6t^2+3t+1-\varepsilon(G)$ constructed in the proof of Theorem \ref{teo4} i); and let $H$ be the graph of order $n$ obtained by adding $n-(6t^2+3t+1-\varepsilon(G))$ vertices to $G$. It is clearly that $H$ is an $S$-graph of order $n$ with achromatic number $6t+1$. Since $n<N$, we have that \[\left\lfloor\sqrt{6(n+\varepsilon(S))-47/4}+1/2\right \rfloor<\left\lfloor\sqrt{6(N+\varepsilon(S))-47/4}+1/2\right \rfloor=6(t+1)+1=\psi(H)+6.\]
Therefore \[\psi(H) \geq \left\lfloor\sqrt{6(n+\varepsilon(S))-47/4}+1/2\right \rfloor-5\]
and the result follows.
\end{proof}

\begin{theorem} \label{teo5}
Let $S$ be a non-orientable surface with $c\geq1$ crosscaps and let $\varphi(c)$ be a function such that $\varphi(c)=1$ if $c$ is odd and $\varphi(c)=0$ otherwise.

i) For $n=6t^2+3t+1-\varepsilon(S)+\varphi(c)$, with $n,t\geq 1$, there exists a graph $G$ of order $n$, embeddable in a surface $S$ with $c$ crosscaps, such that 
\[\psi(G)=\psi_s(G)=\left\lfloor \sqrt{6(n+\varepsilon(S))-47/4}+1/2\right\rfloor.\]
ii) For $n\geq10-\varepsilon(S)+\varphi(c)\geq 1$, there exists an embeddable $S$-graph $H$ of order $n$ such that
\[\psi(H) \geq  \left\lfloor \sqrt{6(n+\varepsilon(S))-47/4}-9/2\right\rfloor.\]
\end{theorem}
\begin{proof}
i) Consider the colored optimal planar graph $G$ of $n=6t^2+3t+1$ vertices embedded in the plane which was described in the proof of Theorem \ref{teo1}, see Figure \ref{Fig3}. We proceed to add $c$ crosscaps to the plane in the following way.

Cutting along a circle $C_i$ in the triangle face $u_iv_iw_i$ contained in the subgraph $G_i$ of $G$ and also along a circle $C'_i$ in the face containing $z$, which was the exterior face, for all $i\leq \min\{c/2,t\}$. Now, add a crosshandle connecting both circles. For each $i$, we identify pairs of vertices with the same color, namely, the vertices $u_i$ and $v_i$ of the exterior face with the colored vertices $u_i$ and $v_i$ of the triangle face of $G_i$, respectively, which is possible moving them through the crosshandle, see Figure \ref{Fig10}.
\begin{figure}[!htbp]
\begin{center}
\includegraphics{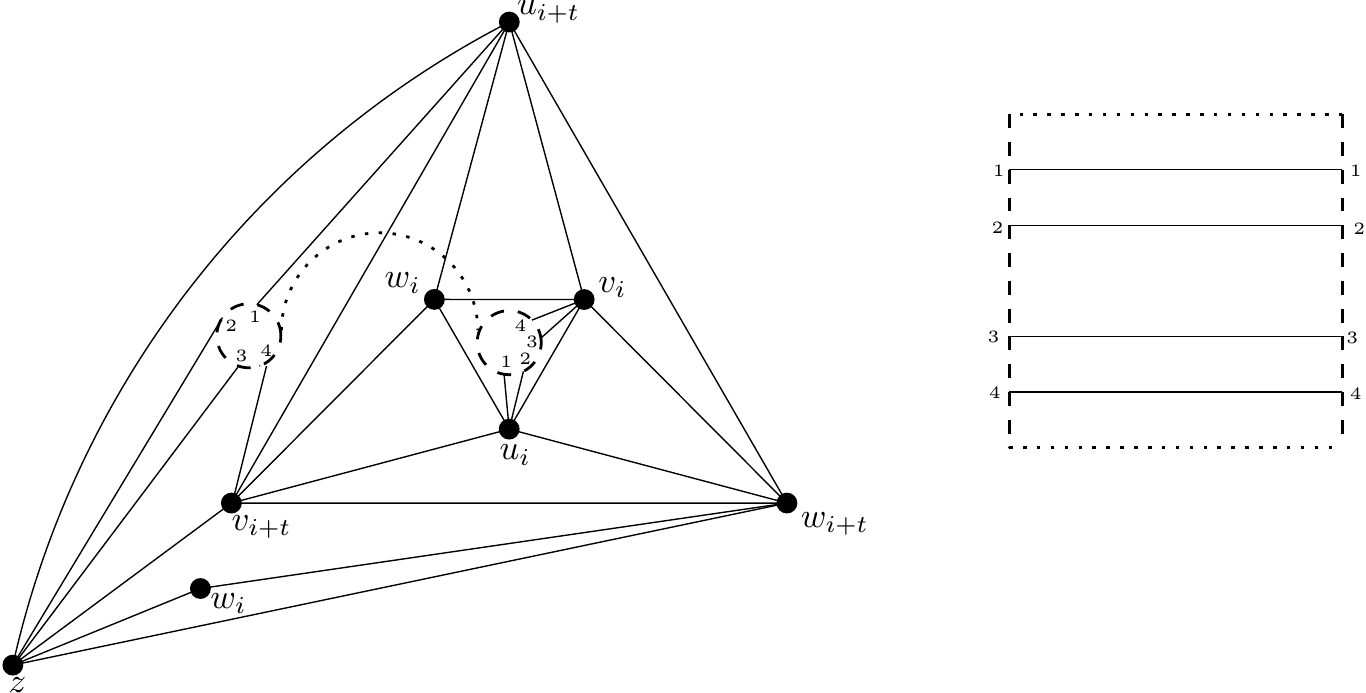}
\caption{Adding a crosshandle and identifying colored vertices $u_i$ and $v_i$.}\label{Fig10}
\end{center}
\end{figure}

If $\varphi(c)=1$, we add a crosscap to some face. We obtain a surface with $c$ crosscaps in which is embedded the graph $G$ of $n=6t^2+3t+1-\varepsilon(S)+\varphi(c)$ vertices and colored with $6t+1$ colors, see Figure \ref{Fig9}.

On the other hand, for $n=6t^2+3t+1-\varepsilon(S)+\varphi(c)$, $\left\lfloor\sqrt{6(n+\varepsilon(S))-47/4}+1/2\right \rfloor$ is equal to \[\left\lfloor\sqrt{36t^2+18t+6\varphi(c)-23/4}+1/2\right \rfloor.\]
Since $\sqrt{36t^2+18t+6\varphi(c)-23/4}+1/2<6t+2$ because $36t^2+18t+6\varphi(c)-23/4\leq(6t+3/2)^2=36t^2+18t+9/4$ and $6t+1\leq \sqrt{36t^2+18t+6\varphi(c)-23/4}+1/2$ because $36t^2+6t+1/4=(6t+1/2)^2\leq36t^2+18t+6\varphi(c)-23/4,$ it follows that \[\left\lfloor\sqrt{36t^2+18t+6\varphi(c)-23/4}+1/2\right \rfloor=6t+1.\]

Finally, we have \[\left\lfloor \sqrt{6(n+\varepsilon(S))-47/4}+1/2\right\rfloor\leq\psi(G)\leq\psi_s(G)\leq\left\lfloor \sqrt{6(n+\varepsilon(S))-47/4}+1/2\right\rfloor\]and the result follows.

ii)
Assume that $6t^2+3t+1-\varepsilon(G)+\varphi(c)\leq n<N=6(t+1)^2+3(t+1)+1-\varepsilon(G)+\varphi(c)$ for some natural number $t$. Let $G$ be the $S$-graph of order $6t^2+3t+1-\varepsilon(G)+\varphi(c)$ constructed in the proof of Theorem \ref{teo4} i); and let $H$ be the graph of order $n$ obtained by adding $n-(6t^2+3t+1-\varepsilon(G)+\varphi(c))$ vertices to $G$. It is clearly that $H$ is an $S$-graph of order $n$ with achromatic number $6t+1$. Since $n<N$, we have that \[\left\lfloor\sqrt{6(n+\varepsilon(S))-47/4}+1/2\right \rfloor<\left\lfloor\sqrt{6(N+\varepsilon(S))-47/4}+1/2\right \rfloor=6(t+1)+1=\psi(H)+6.\]
Therefore \[\psi(H) \geq \left\lfloor\sqrt{6(n+\varepsilon(S))-47/4}+1/2\right \rfloor-5\]
and the result follows.
\end{proof}

\section{The Platonic graphs} \label{The Platonic graphs}
In this section, we show the exact values of the achromatic numbers for the Platonic graphs.
 
We recall that a \emph{Platonic graph} is the skeleton of a Platonic solid. The five Platonic graphs are the tetrahedral graph, cubical graph, octahedral graph, dodecahedral graph, and icosahedral graph. Table \ref{Tab1} shows their order, regularity and achromatic numbers in the columns 2, 3, 4 and 5, respectively. The last column explains how to get the upper bound $k_0$.
\begin{table}[!htbp]
\begin{center}
\begin{tabular}{|l|l|l|l|l|l|}
\hline 
Platonic graphs & $n$ & $r$ & $\psi$ & $\psi_s$ & Upper bound $k_{0}$\tabularnewline
\hline 
\hline 
Tetrahedral & $4$ & $3$ & $4$ & $4$ & $k_0=n=4$ since it is $K_4$.\tabularnewline
\hline 
\multirow{2}{*}{Cubical} & \multirow{2}{*}{$8$} & \multirow{2}{*}{$3$} & \multirow{2}{*}{$4$} & \multirow{2}{*}{$4$} & If a complete coloring contains a color class of size $1$, \tabularnewline
&  &  &  &  &  $k_{0}=r+1=4$, otherwise, $k_{0}=n/2=4$ then $k_{0}=4$.\tabularnewline
\hline Octahedral & $6$ & $4$ & $3$ & $4$ & It is the line graph of $K_{4}$, see \cite{MR3249588,MR3774452}.\tabularnewline
\hline 
\multirow{2}{*}{Dodecahedral} & \multirow{2}{*}{$20$} & \multirow{2}{*}{$3$} & \multirow{2}{*}{$7$} & \multirow{2}{*}{$7$} & If a complete coloring contains a color class of size $2$, \tabularnewline
&  &  &  &  & $k_{0}=1+2r=7$, otherwise,  $k_{0}=\left\lfloor n/3\right\rfloor =6$ then $k_{0}=7$.\tabularnewline
\hline 
\multirow{2}{*}{Icosahedral} & \multirow{2}{*}{$12$} & \multirow{2}{*}{$5$} & \multirow{2}{*}{\textbf{$6$}} & \multirow{2}{*}{$6$} & If a complete coloring contains a color class of size $1$, 
\tabularnewline
&  &  &  &  & $k_{0}=1+r=6$, otherwise, $k_{0}=n/2=6$ then $k_{0}=6$.\tabularnewline
\hline 
\end{tabular}
\caption{Exact values of the achromatic numbers of the Platonic graphs.}\label{Tab1}
\end{center}
\end{table}

The lower bound is illustrated in Figure \ref{Fig11} given by complete colorings of the Platonic graphs. The octahedral graph is twice because it has different values for the achromatic and pseudoachromatic numbers, respectively.
\begin{figure}[!htbp]
\begin{center}
\includegraphics{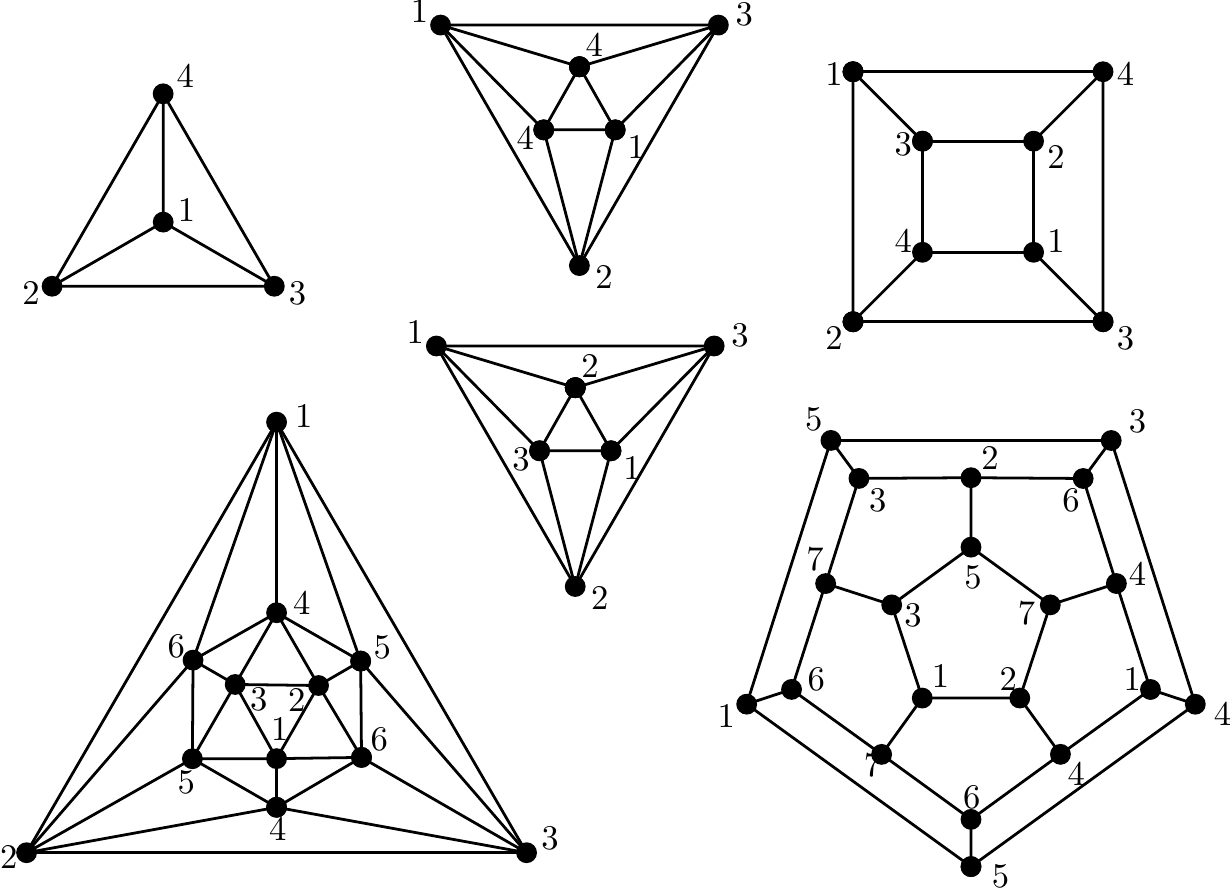}
\caption{Complete colorings of the Platonic graphs with the largest possible number of colors.}\label{Fig11}
\end{center}
\end{figure}
\section*{Acknowledgments}
The authors wish to thank the anonymous referees of this paper for their suggestions and remarks.

Part of the work was done during the I Taller de Matem{\' a}ticas Discretas, held at Campus-Juriquilla, Universidad Nacional Aut{\' o}noma de M{\' e}xico, Quer{\' e}taro City, Mexico on July 28--31, 2014. Part of the results of this paper was announced at the XXX Coloquio V{\' i}ctor Neumann-Lara de Teor{\' i}a de Gr{\' a}ficas, Combinatoria y sus Aplicaciones in Oaxaca, Mexico on March 2--6, 2015.

G. A-P. partially supported by CONACyT-Mexico, grant 282280; and PAPIIT-Mexico, grants IN106318, IN104915, IN107218. C. R-M. partially supported by PAPIIT-Mexico, grant IN107218 and PAIDI/007/18.

\end{document}